\documentclass[12pt,a4paper]{article}
\usepackage{indentfirst}
\usepackage{amsmath,amssymb,amsfonts,amsthm,graphics}
\usepackage{makeidx}
\newtheorem{thm}{Theorem}[section]

\newtheorem{prob}{Problem}[section]
\newtheorem{lem}{Lemma}[section]

\theoremstyle{definition}

\def\-{\mbox{--}}

\begin{document}
\title{\Large\bf A sharp upper bound for the rainbow 2-connection number
of 2-connected graphs\footnote{Supported by NSFC No.11071130.}}
\author{\small  Xueliang~Li, Sujuan~Liu\\
\small Center for Combinatorics and LPMC-TJKLC\\
\small Nankai University, Tianjin 300071, China\\
\small lxl@nankai.edu.cn; sjliu0529@126.com}
\date{}
\maketitle
\begin{abstract}
A path in an edge-colored graph is called {\em rainbow} if no two
edges of it are colored the same. For an $\ell$-connected graph $G$
and an integer $k$ with $1\leq k\leq \ell$, the {\em rainbow
$k$-connection number} $rc_k(G)$ of $G$ is defined to be the minimum
number of colors required to color the edges of $G$ such that every
two distinct vertices of $G$ are connected by at least $k$
internally disjoint rainbow paths. Fujita et. al. proposed a problem
that what is the minimum constant $\alpha>0$ such that for all
2-connected graphs $G$ on $n$ vertices, we have $rc_2(G)\leq \alpha
n$. In this paper, we prove that $\alpha=1$ and $rc_2(G)=n$ if and
only if $G$ is a cycle of order $n$, settling down this problem.



\end{abstract}

\section{Introduction}

All graphs considered in this paper are simple, finite and
undirected. We follow the terminology and notation of Bondy and
Murty \cite{bondy2008graph}. A path in an edge-colored graph is
called {\em rainbow} if every two edges on it have distinct colors.
Let $G$ be an edge-colored $\ell$-connected graph, where $\ell$ is a
positive integer. For $1\leq k\leq \ell$, $G$ is {\em rainbow
$k$-connected} if every pair of distinct vertices of $G$ are
connected by at least $k$ internally disjoint rainbow paths. The
minimum number of colors required to color the edges of $G$ to make
$G$ rainbow $k$-connected is called the {\em rainbow $k$-connection
number} of $G$, denoted by $rc_k(G)$. Particularly, $rc_1(G)$ is
equal to $rc(G)$, the rainbow connection number. For more results on
this topic, see a recent book by Li and Sun \cite{Li2012connected}.

A graph $G$ is {\em minimally $k$-connected} if $G$ is $k$-connected
but $G-e$ is not $k$-connected for every $e\in E(G)$. Let $G^\prime$
be a subgraph of a graph $G$. An {\em ear} of $G^\prime$ in $G$ is a
nontrivial path in $G$ whose end vertices lie in $G^\prime$ but
whose internal vertices are not. An {\em ear decomposition} of a
2-connected graph $G$ is a sequence $G_0, G_1, \cdots, G_k$ of
2-connected subgraphs of $G$ such that (1) $G_0$ is a cycle of $G$;
(2) $G_i=G_{i-1}\bigcup P_{i-1}(1\leq i\leq k)$, where $P_{i-1}$ is
an ear of $G_{i-1}$ in $G$; (3) $G_{i-1}(1\leq i\leq k)$ is a proper
subgraph of $G_i$; (4) $G_k=G$. It is obvious that every graph $G_i$
in an ear decomposition is 2-connected. Two paths $P^\prime,
P^{\prime\prime}$ from $v_i$ to $v_j$ are internally disjoint if
$V(P^\prime)\bigcap V(P^{\prime\prime})=\{v_i, v_j\}$. For three
distinct vertices $v^\prime, v_1^{\prime\prime},
v_2^{\prime\prime}$, the paths $P^\prime$ and $P^{\prime\prime}$
from $v^\prime$ to $v_1^{\prime\prime}$ and $v_2^{\prime\prime}$,
respectively, are internally disjoint if $V(P^\prime)\bigcap
V(P^{\prime\prime})=\{v^\prime\}$. Two paths $P^\prime$ and
$P^{\prime\prime}$ are disjoint if $V(P^\prime)\bigcap
V(P^{\prime\prime})=\emptyset$.

The concept of rainbow $k$-connection number $rc_k(G)$ was
introduced by Chartrand et. al. \cite{Chartrand2008graph,
Chartrand2009graph}. It was shown in \cite{chakraborty2009hardness}
that computing the rainbow connection number of a graph is NP-hard.
Hence, bounds on rainbow connection number for graphs have been a
subject of investigation. There are some results in this direction.
For a connected graph $G$, $rc(G)\leq n-1$ in
\cite{caro2008rainbow}. An upper bound for a connected graph with
minimum degree $\delta$ is $3n/(\delta+1)+3$ in
\cite{Chandran2011dominating}. If $G$ is a 2-connected graph of
order $n$, then $rc(G)\leq\lceil\frac{n}{2}\rceil$ and
$rc(C_n)=\lceil \frac{n}{2}\rceil$, where $C_n$ is an $n$-vertex
cycle in \cite{Li2011connected}. An easy observation is that
$rc_2(C_n)=n$. In \cite{Fujita2011}, the authors proved the
following theorem and proposed a problem.

\begin{thm}\cite{Fujita2011}\label{thm1}
If $\ell\geq2$ and $G$ is an $\ell$-connected graph of order $n\geq \ell+1$,
then $rc_2(G)\leq(\ell+1)n/\ell$.
\end{thm}

\begin{prob}\cite{Fujita2011}\label{problem}
What is the minimum constant $\alpha>0$ such that for all 2-connected graphs $G$ on $n$
vertices, we have $rc_2(G)\leq \alpha n$?
\end{prob}

In a published version of \cite{Fujita2011}, they stated the
following theorem and problem.

\begin{thm}\cite{Fujita2011Dense}\label{thm11}
If $G$ is a $2$-connected graph of order $n\geq 3$, then
$rc_2(G)\leq {3n}/2$.
\end{thm}

\begin{prob}\cite{Fujita2011Dense}\label{problem11}
For $1\leq k\leq \ell$, derive a sharp upper bound for $rc_k(G)$, if
$G$ is an $\ell$-connected graph on $n$ vertices. Is there a
constant $\alpha=\alpha(k,\ell)$ such that $rc_k(G) \leq \alpha n$ ?
\end{prob}

Problem \ref{problem} is restated in \cite{Li2012connected}. From
Theorem \ref{thm11} and $rc_2(C_n)=n$, it is obvious that $1\leq
\alpha\leq 3/2$. For a 2-connected series-parallel graph $G$, the
authors of \cite{Fujita2011, Fujita2011Dense} showed that
\begin{thm}\cite{Fujita2011, Fujita2011Dense}
If $G$ is a 2-connected series-parallel graph on $n$ vertices, then
$rc_2(G)\leq n$.
\end{thm}

In this paper, we will show that the above result holds for general
2-connected graphs.
\begin{thm} \label{thm0}
If $G$ is a 2-connected graph on $n$ vertices, then $rc_2(G)\leq n$
with equality if and only if $G$ is a cycle of order $n$. Therefore,
the constant $\alpha=1$ in Problem \ref{problem}.
\end{thm}

The following classic results on minimally 2-connected graphs are
needed in the sequel.

\begin{thm}\cite{Bollobas1978graph}\label{thm2}
Let $G$ be a minimally 2-connected graph that is not a cycle. Let $D\subset V(G)$ be the set
of vertices of degree two. Then $F=G-D$ is a forest with at least two components. A component
$P$ of $G[D]$ is a path and the end vertices of $P$ are not jointed to the same tree of the
forest $F$.
\end{thm}

\begin{thm}\cite{Bollobas1978graph}\label{thm3}
Every 2-connected subgraph of a minimally 2-connected graph is minimally 2-connected.
\end{thm}

\begin{thm}\cite{bondy2008graph}\label{thm4}
Every nontrivial tree has at least two leaves.
\end{thm}

\section{Main results}

We first give a lemma, which will be used next.

\begin{lem}\label{lem1}
Let $G$ be a minimally 2-connected graph, and $G$ is not a cycle. Then $G$ has an ear
decomposition $G_0, G_1, \cdots, G_t(t\geq1)$ satisfying the following conditions:

\noindent (1) $G_i=G_{i-1}\bigcup P_{i-1}(1\leq i\leq t)$, where $P_{i-1}$ is an ear
of $G_{i-1}$ in $G$ and at least one vertex of $P_{i-1}$ has degree two in $G$;

\noindent (2) each of the two internally disjoint paths in $G_0$ between the end vertices
of $P_0$ has at least one vertex of degree two in $G$.
\end{lem}

\begin{proof} We first construct a sequence of 2-connected subgraphs of $G$. Let $D\subset
V(G)$ be the set of vertices of degree two in $G$. Let $G_0$ be a
cycle of $G$ which contains as many vertices of $D$ as possible. If
$D\backslash V(G_0)\neq\emptyset$, then choose a vertex $v_0\in
D\backslash V(G_0)$. Since $G$ is 2-connected, from Menger's Theorem
there exist two internally disjoint paths $P^\prime,
P^{\prime\prime}$ from $v_0$ to two distinct vertices of $G_0$.
Hence $P_0=P^\prime\bigcup P^{\prime\prime}$ is an ear of $G_0$
which contains a vertex $v_0$ in $D$. Let $G_1=G_0\bigcup P_0$. If
$D\backslash V(G_1)\neq\emptyset$, then we continue the procedure.
After finite steps, we get a sequence $G_0, G_1,\cdots, G_t(t\geq1)$
of 2-connected subgraphs of $G$ such that $D\backslash
V(G_t)=\emptyset$ and $G_i=G_{i-1}\bigcup P_{i-1}(1\leq i\leq t)$,
where $P_{i-1}$ is an ear of $G_{i-1}$ containing at least one
vertex in $D$. If $G_t=G$, then from the procedures of construction,
the sequence $G_0, G_1,\cdots, G_t(t\geq1)$ is an ear decomposition
of $G$ satisfying condition (1).

We first show that $G_t=G$. Suppose on the contrary that $G_t\neq
G$, i.e., $G_t$ is a proper 2-connected subgraph of $G$. Since $G$
is minimally 2-connected, from Theorem \ref{thm3} we have
$V(G)\backslash V(G_t)\neq\emptyset$. From Theorem \ref{thm2}, $G-D$
is a forest. Since $D\subseteq V(G_t)$, $F=G-V(G_t)\subset G-D$ is
also a forest with $|F|\geq1$. Let $T$ be a component of $F$ with
$|T|\geq1$. Then $T$ is a tree. If $|T|=1$ and $V(T)=\{v\}$, then
there exist three distinct vertices $v_1, v_2, v_3$ in $G_t$ such
that $vv_j\in E(G)(1\leq j\leq 3)$. Let $G^\prime=(V^\prime,
E^\prime)$, where $V^\prime=V(G_t)\bigcup\{v\}$ and
$E^\prime=E(G_t)\bigcup\{vv_j:1\leq j\leq3\}$. So $G^\prime$ is a
2-connected subgraph of $G$. Since $G^\prime-vv_3$ is also
2-connected, $G^\prime$ is not minimally 2-connected which
contradicts to Theorem \ref{thm3}. Suppose $|T|\geq2$. From Theorem
\ref{thm4}, $T$ has at least two leaves (say $v^\prime,
v^{\prime\prime}$). Since $v^\prime, v^{\prime\prime}\notin D$ and
$d_T(v^\prime)=d_T(v^{\prime\prime})=1$, there exist four vertices
$v_i(1\leq i\leq4, v_1\neq v_2$ $v_3\neq v_4)$ in $G_t$ such that
$v^\prime v_1, v^\prime v_2, v^{\prime\prime}v_3,
v^{\prime\prime}v_4\in E(G)$. Let $P$ be the path from $v^\prime$ to
$v^{\prime\prime}$ in $T$. Then $G^\prime=(V^\prime, E^\prime)$,
where $V^\prime=V(G_t)\bigcup V(P)$ and $E^\prime=E(G_t)\bigcup
E(P)\bigcup\{v^\prime v_1, v^\prime v_2, v^{\prime\prime}v_3,
v^{\prime\prime}v_4\}$ is a 2-connected subgraph of $G$. Since
$G^\prime-vv_1$ is also 2-connected, $G^\prime$ is not minimally
2-connected which contradicts to Theorem \ref{thm3}. Therefore,
$G_t=G$.

Now we show that the ear decomposition $G_0, G_1, \cdots,
G_t(t\geq1)$ of $G$ satisfies condition (2). Denote by $P^\prime$
and $P^{\prime\prime}$ the two internally disjoint paths in $G_0$
between the two end vertices of $P_0$. Suppose on the contrary that
one of $P^\prime$ and $P^{\prime\prime}$(say $P^\prime$) has no
vertex of degree two in $G$, i.e., $V(P^\prime) \bigcap
D=\emptyset$. From the procedure of construction, $V(P_0)\bigcap
D\neq\emptyset$. Hence, $P^{\prime\prime}\bigcup P_0$ is a cycle of
$G$, which contains more vertices in $D$ than $G_0$, a
contradiction. Therefore, the ear decomposition $G_0, G_1, \cdots,
G_t(t\geq1)$ of $G$ satisfies condition (2).
\end{proof}

For convenience, we give some more notations. If $c$ is an
edge-coloring of a graph $G$, then $c(G)$ denotes the set of colors
appearing in $G$. Write $|G|$ for the order of a graph $G$. If $P$
is a path and $v_i, v_j\in V(P)$, then $v_iPv_j$ denotes the segment
of $P$ from $v_i$ to $v_j$.

\begin{lem}\label{lem2}
Let $G$ ba a minimally 2-connected graph of order $n\geq 3$. If $G$
is not a cycle, then $rc_2(G)\leq n-1$.
\end{lem}

\begin{proof}
Let $G$ be a minimally 2-connected graph of order $n$ and $G$ is not
a cycle. We will prove the result by giving an edge-coloring of $G$
with $n-1$ colors which makes $G$ rainbow 2-connected. From Lemma
\ref{lem1}, $G$ has an ear decomposition $G_0, G_1, \cdots,
G_t(t\geq1)$ satisfying the two conditions in Lemma \ref{lem1}. Let
$D\subset V(G)$ be the set of vertices of degree two in $G$ and
$\overline{D}=V(G)\backslash D$. In the following, for every graph
$G_i(1\leq i\leq t)$ we will define an edge-coloring $c_i$ of $G_i$
with $|G_i|-1$ colors and a map $f_i$ from $\overline{D}\bigcap
V(G_i)$ to $c_i(G_i)$ satisfying the following conditions:

\noindent(A1): $G_i$ is rainbow 2-connected;

\noindent(A2): for any three distinct vertices $v^\prime,
v_1^{\prime\prime}, v_2^{\prime\prime}\in V(G_i)$, $G_i$ has two
internally disjoint rainbow paths $P^\prime$ and $P^{\prime\prime}$
from $v^\prime$ to $v_1^{\prime\prime}$ and $v_2^{\prime\prime}$,
respectively;

\noindent(A3): for any four distinct vertices $v_1^\prime,
v_2^\prime, v_1^{\prime\prime}, v_2^{\prime\prime}\in V(G_i)$, $G_i$
has two disjoint rainbow paths $P^\prime$ from $v_1^\prime$ to one
of $v_1^{\prime\prime}, v_2^{\prime\prime}$(say
$v_1^{\prime\prime}$) and $P^{\prime\prime}$ from $v_2^\prime$ to
the other vertex $v_2^{\prime\prime}$;

\noindent(A4): $f_i$ is injective, i.e., for any two distinct
vertices $v^\prime, v^{\prime\prime}\in\overline{D}\bigcap V(G_i)$,
$f_i(v^\prime)\neq f_i(v^{\prime\prime})$;

\noindent(A5): for any vertex $v\in\overline{D}\bigcap V(G_i)$, the
color $f_i(v)$ appears exactly once in $c_i$ and the edge colored by
$f_i(v)$ in $G_i$ is incident with $v$.

We define $c_i$ and $f_i$ of $G_i(1\leq i\leq t)$ by induction.
First, consider the graph $G_1=G_0\bigcup P_0$. Without loss of
generality, suppose that $G_0=v_1, v_2, \cdots, v_s$ and $P_0=v_1,
v_{s+1},$ $v_{s+2}, \cdots, v_\ell, v_p(\ell>s)$, where $G_0$ is a
cycle, $P_0$ is a path and $V(G_0)\bigcap V(P_0)=\{v_1, v_p\}(3\leq
p\leq s-1)$. Since the ear decomposition $G_0, G_1, \cdots,
G_t(t\geq 1)$ of $G$ satisfies the two conditions in Lemma
\ref{lem1}, there exist three vertices $v_{p_1}, v_{p_2}, v_{p_3}\in
D(1<p_1<p<p_2\leq s<p_3\leq\ell)$ in $G_1$. Define an edge-coloring
$c_1$ of $G_1$ by $c_1(v_jv_{j+1}) =x_j$ if $1\leq j\leq s-1$ or
$s+1\leq j\leq\ell-1$; $c_1(v_sv_1)=c_1(v_\ell v_p)=x_s$ and
$c_1(v_1v_{s+1})=x_p$, where $x_1, x_2,\cdots, x_{\ell-1}$ are
distinct colors. It is obvious that $c_1$ uses $|G_1|-1$ colors.
Define a map $f_1: \overline{D}\bigcap V(G_1) \rightarrow c_1(G_1)$
by $f_1(v_j)=x_j$ if $v_j\in \overline{D}\bigcap V(G_1)$ and $1\leq
j<p_1, p+1\leq j<p_2$ or $s+1\leq j<p_3$ and $f_1(v_j)=x_{j-1}$ if
$v_j\in \overline{D}\bigcap V(G_1)$ and $p_1< j\leq p, p_2<j\leq s$
or $p_3<j\leq\ell$. It can be checked that $c_1$ and $f_1$ satisfy
the above conditions (A1)-(A5).

If $t=1$, then $c_1$ is the rainbow 2-connected edge-coloring of $G$
with $n-1$ colors. Consider the case $t\geq 2$. Assume that we have
defined $c_{i-1}$ and $f_{i-1}$ of $G_{i-1}(2\leq i\leq t)$
satisfying conditions (A1)-(A5) and the edge-coloring $c_{i-1}$ of
$G_{i-1}$ uses $|G_{i-1}|-1$ colors. Now consider the graph
$G_i=G_{i-1}\bigcup P_{i-1}$. Suppose that $P_{i-1}=v_1, v_2,\cdots,
v_q(q\geq3)$, where $V(G_{i-1})\bigcup V(P_{i-1})= \{v_1, v_q\}$. It
is obvious that $v_1, v_q\in\overline{D}\bigcap V(G_{i-1})$. Define
an edge-coloring $c_i$ of $G_i$ by $c_i(e)=c_{i-1}(e)$ for $e\in
E(G_{i-1})$, $c_i(v_{q-1}v_q)=f_{i-1}(v_1)$ and
$c_i(v_jv_{j+1})=y_j(1\leq j\leq q-2)$, where $y_1, y_2, \cdots,
y_{q-2}$ are distinct new colors. It is clear that $c_i$ uses
$|G_i|-1$ colors. From condition (1) of the Lemma \ref{lem1}, there
exists a vertex $v_{q_0}\in D(2\leq q_0\leq q-1)$ in $P_{i-1}$.
Define a map $f_i: \overline{D}\bigcap V(G_i)\rightarrow c_i(G_i)$
as follows: $f_i(v)=f_{i-1}(v)$ for $v\in[\overline{D}\bigcap
V(G_{i-1})] \backslash\{v_1\}$, $f_i(v_j)=y_j$ for
$v_j\in\overline{D}\bigcap V(v_1P_{i-1}v_{q_0-1})$ and
$f_i(v_j)=y_{j-1}$ for $v_j\in\overline{D}\bigcap
V(v_{q_0+1}P_{i-1}v_{q-1})$. The edge-coloring $c_i$ of $G_i$ has
the following two properties.

\noindent(B1): There exists a rainbow path $P_{i-1}^\prime$ from
$v_1$ to $v_q$ in $G_{i-1}$ such that the color $f_{i-1}(v_1)$ does
not appear on it. In fact, since $G_{i-1}$ is rainbow 2-connected,
there are two internally disjoint rainbow paths in $G_{i-1}$
connecting $v_1, v_q$. Since the map $f_{i-1}$ satisfies condition
(A5), the color $f_{i-1}(v_1)$ appears exactly once in $G_{i-1}$. So
$f_{i-1}(v_1)$ does not appear on one of the two rainbow paths,
denoted by $P_{i-1}^\prime$, from $v_1$ to $v_q$.

\noindent(B2): Since $c_{i-1}$ and $f_{i-1}$ satisfy condition (A5),
the color $f_{i-1}(v_1)$ does not appear on any path in $G_{i-1}$
which does not contain $v_1$.

We will show that $c_i$ and $f_i$ satisfy conditions (A1)-(A5).

(I). Consider any two distinct vertices $v^\prime,
v^{\prime\prime}\in V(G_i)$. If $v^\prime, v^{\prime\prime}\in
V(G_{i-1})$, there exist two internally disjoint rainbow paths
connecting them in $G_{i-1}$, which are also rainbow paths in $G_i$
according to the definition of $c_i$. Assume $v^\prime,
v^{\prime\prime}\in V(P_{i-1})$. From property (B1),
$P_{i-1}^\prime\bigcup P_{i-1}$ is a cycle whose colors are
distinct. Hence there are two internally disjoint rainbow paths from
$v^\prime$ to $v^{\prime\prime}$ on the cycle $P_{i-1}^\prime\bigcup
P_{i-1}$. Assume $v^\prime\in V(G_{i-1})\backslash \{v_1, v_q\}$ and
$v^{\prime\prime}\in V(P_{i-1})\backslash \{v_1, v_q\}$. Since
$c_{i-1}$ satisfies condition (A2), there exist two internally
disjoint rainbow paths $P^\prime$ and $P^{\prime\prime}$ in
$G_{i-1}$ from $v^\prime$ to $v_1$ and $v_q$, respectively. From
property (B2) and $v_1\notin V(P^{\prime\prime})$, we have
$f_{i-1}(v_1)\notin c_i(P^{\prime\prime})$. So $v^\prime P^\prime
v_1P_{i-1}v^{\prime\prime}$ and $v^\prime P^{\prime\prime}
v_qP_{i-1}v^{\prime\prime}$ are two internally disjoint rainbow
paths from $v^\prime$ to $v^{\prime\prime}$ in $G_i$.  Therefore,
$G_i$ is rainbow 2-connected.

(II). Consider any three distinct vertices $v^\prime,
v_1^{\prime\prime}, v_2^{\prime\prime} \in V(G_i)$. If $v^\prime,
v_1^{\prime\prime}, v_2^{\prime\prime}\in V(G_{i-1})$, then from
condition (A2) of $c_{i-1}$ and the definition of $c_i$, there exist
two internally disjoint rainbow paths $P^\prime$ and
$P^{\prime\prime}$ in $G_{i-1}$ from $v^\prime$ to
$v_1^{\prime\prime}$ and $v_2^{\prime\prime}$, respectively. If
$v^\prime, v_1^{\prime\prime}, v_2^{\prime\prime}\in V(P_{i-1})$,
from property (B1) there exist two internally disjoint rainbow paths
on the cycle $P_{i-1}^\prime\bigcup P_{i-1}$ from $v^\prime$ to
$v_1^{\prime\prime}$ and $v_2^{\prime\prime}$, respectively.
Consider the case that $v^\prime, v_1^{\prime\prime}\in
V(G_{i-1})\backslash\{v_1\}$ and $v_2^{\prime\prime}\in
V(P_{i-1})\backslash\{v_q\}$. From condition (A2) of $c_{i-1}$,
there exist two internally disjoint rainbow paths $P^\prime$ and
$P^{\prime\prime}$ in $G_{i-1}$ from $v^\prime$ to
$v_1^{\prime\prime}$ and $v_1$, respectively. So $P^\prime$ and
$v^\prime P^{\prime\prime}v_1P_{i-1}v_2^{\prime\prime}$ are two
internally disjoint rainbow paths in $G_i$ from $v^\prime$ to
$v_1^{\prime\prime}$ and $v_2^{\prime\prime}$, respectively.
Consider the case that $v^\prime\in V(G_{i-1})\backslash\{v_1,
v_q\}$ and $v_1^{\prime\prime}, v_2^{\prime\prime}\in V(P_{i-1})$.
Without loss of generality, $v_1^{\prime\prime}$,
$v_2^{\prime\prime}$ appear on $P_{i-1}$ in this order. From
condition (A2) of $c_{i-1}$, there exist two internally disjoint
rainbow paths $P^\prime$ and $P^{\prime\prime}$ in $G_{i-1}$ from
$v^\prime$ to $v_1$ and $v_q$, respectively. From property (B2) and
$v_1\notin V(P^{\prime\prime})$, we have $f_{i-1}(v_1)\notin
c_i(P^{\prime\prime})$. So $v^\prime P^\prime
v_1P_{i-1}v_1^{\prime\prime}$ and $v^\prime P^{\prime\prime}
v_qP_{i-1}v_2^{\prime\prime}$ are two internally disjoint rainbow
paths in $G_i$ from $v^\prime$ to $v_1^{\prime\prime}$ and
$v_2^{\prime\prime}$, respectively. Consider the case that
$v_1^{\prime\prime}, v_2^{\prime\prime}\in V(G_{i-1})
\backslash\{v_1\}$ and $v^\prime\in V(P_{i-1})\backslash\{v_q\}$.
From condition (A3) of $c_{i-1}$, there exist two disjoint rainbow
paths $P^\prime$ from $v_1^{\prime\prime}$ to one of $v_1, v_q$ (say
$v_1$) and $P^{\prime\prime}$ from $v_2^{\prime\prime}$ to the other
vertex $v_q$. If $v_2^{\prime\prime}=v_q$, then $P^{\prime\prime}=
v_2^{\prime\prime}$. From property (B2), $v^\prime
P_{i-1}v_1P^\prime v_1^{\prime\prime}$ and $v^\prime
P_{i-1}v_qP^{\prime\prime}v_2^{\prime\prime}$ are two internally
disjoint rainbow paths from $v^\prime$ to $v_1^{\prime\prime}$ and
$v_2^{\prime\prime}$, respectively. Consider the case that
$v_2^{\prime\prime}\in V(G_{i-1}) \backslash\{v_1\}$ and
$v_1^{\prime\prime}, v^\prime\in V(P_{i-1})$. Without loss of
generality, $v_1^{\prime\prime}$, $v^\prime$ appear on $P_{i-1}$ in
this order. Since the color $f_{i-1}(v_1)$ appears exactly once in
$G_{i-1}$, one of the two internally disjoint rainbow paths in
$G_{i-1}$ from $v_q$ to $v_2^{\prime\prime}$, denoted by $P^\prime$,
does not contain the edge colored by $f_{i-1}(v_1)$, i.e.,
$f_{i-1}(v_1)\notin c_i(P^\prime)$. So $v^\prime
P_{i-1}v_1^{\prime\prime}$ and $v^\prime P_{i-1}v_qP^\prime
v_2^{\prime\prime}$ are two internally disjoint rainbow paths in
$G_i$ from $v^\prime$ to $v_1^{\prime\prime}$ and
$v_2^{\prime\prime}$, respectively. Therefore, $c_i$ satisfies
condition (A2).

(III). Consider any four distinct vertices $v_1^\prime, v_2^\prime,
v_1^{\prime\prime}, v_2^{\prime\prime}\in V(G_i)$. If $v_1^\prime,
v_2^\prime, v_1^{\prime\prime}, v_2^{\prime\prime}\in V(G_{i-1})$,
then there exist two required disjoint rainbow paths in $G_{i-1}$
from condition (A3) of $c_{i-1}$ and the definition of $c_i$. If
$v_1^\prime, v_2^\prime, v_1^{\prime\prime}, v_2^{\prime\prime}\in
V(P_{i-1})$, then there exist two required disjoint rainbow paths on
the cycle $P_{i-1}^\prime\bigcup P_{i-1}$ from property (B1).
Consider the case that $v_1^\prime, v_2^\prime,
v_1^{\prime\prime}\in V(G_{i-1})\backslash\{v_1\}$ and
$v_2^{\prime\prime}\in V(P_{i-1})\backslash\{v_q\}$. From condition
(A3) of $c_{i-1}$, there exist two disjoint rainbow paths $P^\prime$
from $v_1^\prime$ to one of $v_1, v_1^{\prime\prime}$ (say
$v_1^{\prime\prime}$) and $P^{\prime\prime}$ from $v_2^\prime$ to
the other vertex $v_1$ in $G_{i-1}$. Then $P^\prime$ and $v_2^\prime
P^{\prime\prime}v_1P_{i-1}v_2^{\prime\prime}$ are two required
disjoint rainbow paths in $G_i$. Consider the case that $v_1^\prime,
v_2^\prime\in V(G_{i-1})\backslash\{v_1\}$ and $v_1^{\prime\prime},
v_2^{\prime\prime}\in V(P_{i-1})\backslash\{v_q\}$. Without loss of
generality, $v_1^{\prime\prime}, v_2^{\prime\prime}$ appear on
$P_{i-1}$ in this order. From condition (A3) of $c_{i-1}$, there
exist two disjoint rainbow paths $P^\prime$ from $v_1^\prime$ to
$v_1, v_q$ (say $v_1$) and $P^{\prime\prime}$ from $v_2^\prime$ to
the other vertex $v_q$ in $G_{i-1}$. If $v_2^\prime=v_q$, then
$P^{\prime\prime}= v_2^\prime$. Hence $v_1^\prime P^\prime
v_1P_{i-1}v_1^{\prime\prime}$ and $v_2^\prime
P^{\prime\prime}v_qP_{i-1} v_2^{\prime\prime}$ are two required
disjoint rainbow paths in $G_i$. Consider the case that $v_1^\prime,
v_1^{\prime\prime}\in V(G_{i-1})\backslash \{v_1\}$ and $v_2^\prime,
v_2^{\prime\prime}\in V(P_{i-1})\backslash\{v_q\}$. From condition
(A1) of $c_{i-1}$, let $P^\prime$ be a rainbow path from
$v_1^\prime$ to $v_1^{\prime\prime}$ in $G_{i-1}$. Then $P^\prime$
and $v_2^\prime P_{i-1}v_2^{\prime\prime}$ are two required disjoint
rainbow paths in $G_i$. Consider the case that $v_1^\prime\in
V(G_{i-1})\backslash \{v_1\}$ and $v_2^\prime, v_1^{\prime\prime},
v_2^{\prime\prime}\in V(P_{i-1}) \backslash\{v_q\}$. Without loss of
generality, $v_2^\prime, v_2^{\prime\prime}, v_1^{\prime\prime}$
appear on $P_{i-1}$ in this order. From conditions (A1) and (A5) of
$c_{i-1}$ and $f_{i-1}$, there exists one rainbow path $P^\prime$ in
$G_{i-1}$ from $v_1^\prime$ to $v_q$ such that $f_{i-1}(v_1)\notin
c_i(P^\prime)$. Then $v_1^\prime P^\prime
v_qP_{i-1}v_1^{\prime\prime}$ and $v_2^\prime
P_{i-1}v_2^{\prime\prime}$ are two required disjoint rainbow paths
in $G_i$. Therefore, $c_i$ satisfies condition (A3).

(VI). From condition (A4) of $f_{i-1}$ and the definition of $f_i$,
$f_i$ is injective.

(V). From condition (A4) of $f_{i-1}$ and the definition of $f_i$,
$f_i$ satisfies condition (A5).

Therefore, we can get an edge-coloring $c_t$ of $G(=G_t)$ with
$n-1(=|G_t|-1)$ colors which makes $G$ rainbow 2-connected. So
$rc_2(G)\leq n-1$.
\end{proof}

An easy observation is that if $G^\prime$ is a spanning subgraph of
a graph $G$ and $rc_k(G)$ and $rc_k(G^\prime)$ are indeed exist,
then we have $rc_k(G)\leq rc_k(G^\prime)(k\geq 1)$.

Now we are ready to prove our main result Theorem \ref{thm0}.

\noindent {\bf Proof of Theorem \ref{thm0}:} If $G$ is an $n$-vertex
cycle, then we have $rc_2(G)=n$. Hence, to prove the result we only
need to show that for any 2-connected graph $G$ of order $n$ which
is not a cycle, $rc_2(G)\leq n-1$. Let $G$ be such a graph. Consider
the following two cases.

{\bf Case 1}. $G$ is Hamiltonian.

Let $C=v_1, v_2, \cdots, v_n$ be a Hamiltonian cycle of $G$. Since
$G$ is not a cycle, there must be a chordal of $C$ (say $v_1v_j\in
E(G)(3\leq j\leq n-1)$, without loss of generality) in $G$. Then
$G^\prime=(V(G), E(C)\bigcup\{v_1v_j\})$ is a spanning 2-connected
subgraph of $G$. Let $x_1, x_2, \cdots, x_{n-1}$ be $n-1$ distinct
colors. Define an edge-coloring $c$ of $G^\prime$ with $n-1$ colors
as follows: $c(v_1v_2)=c(v_jv_{j+1})=x_1$, $c(v_1v_n)=c(v_{j-1}v_j)
=x_2$ and the other $n-3$ edges of $G^\prime$ are colored by colors
$x_3, \cdots, x_{n-1}$. It can be checked that $G^\prime$ is rainbow
2-connected. From the above observation, $rc_2(G)\leq
rc_2(G^\prime)\leq n-1$.

{\bf Case 2}. $G$ is not Hamiltonian.

Let $G^\prime$ be a spanning minimally 2-connected subgraph of $G$.
Since $G$ is not Hamiltonian, $G^\prime$ is not a cycle. From Lemma
\ref{lem2} and the above observation, we have $rc_2(G)\leq
rc_2(G^\prime)\leq n-1$.

The proof is now complete. $\Box$

\end{document}